\documentclass[12pt,psamsfonts]{amsart}
\usepackage[utf8]{inputenc}

\usepackage{amsmath}
\usepackage{amsthm}
\usepackage{amssymb}
\usepackage{amscd}
\usepackage{amsfonts}
\usepackage{amsbsy}
\usepackage{graphicx}
\usepackage[dvips]{psfrag}
\usepackage{array}
\usepackage{color}
\usepackage{epsfig}
\usepackage{url}

\usepackage{overpic}
\usepackage{epstopdf}

\newcommand{\R}{\ensuremath{\mathbb{R}}}

\newcommand{\CS}{\ensuremath{\mathcal{S}}}

\newcommand{\CF}{\ensuremath{\mathcal{F}}}

\newcommand{\CN}{\ensuremath{\mathcal{N}}}

\newcommand{\CV}{\ensuremath{\mathcal{V}}}

\newcommand{\ov}{\overline}

\newcommand{\al}{\alpha}

\newcommand{\torus}{\ensuremath{\mathbb{T}}}

\newcommand{\sgn}{\mathrm{sign}}

\newcommand{\supp}{\mathrm{supp}}

\newcommand{\sat}{\mathrm{Sat}}

\def\p{\partial}

\newtheorem {theorem} {Theorem}

\newtheorem {proposition} {Proposition}

\newtheorem {remark} {Remark}

\newtheorem {mcor} {Corollary}

\newtheorem{main}{Theorem}

\begin{document}
\renewcommand{\arraystretch}{1.5}

\title[Invariant measures for Filippov systems]
{A note on invariant measures\\ for Filippov systems}

\author[Novaes \& Var\~ao ]
{Douglas D. Novaes and R\'egis Var\~{a}o}

\address{ Departamento de Matem\'{a}tica, Universidade
Estadual de Campinas, Rua S\'{e}rgio Buarque de Holanda, 651, Cidade Universit\'{a}ria Zeferino Vaz, 13083--859, Campinas, SP,
Brazil} \email{ddnovaes@unicamp.br} \email{varao@unicamp.br}



\maketitle

\begin{abstract}
We are interested in Filippov systems which preserve a probability measure on a compact manifold. We define a measure to be invariant for a Filippov system as the natural analogous definition of invariant measure for flows. Our main result concerns Filippov systems which preserve a probability measure equivalent to the volume measure. As a consequence, the volume preserving Filippov systems are the refractive piecewise volume preserving ones. We conjecture that if a Filippov system admits an invariant probability measure, this measure does not see the trajectories where there is a break of uniqueness. We prove this conjecture for Lipschitz differential inclusions. Then, in light of our previous results, we analyze the existence of invariant measures for many examples of Filippov systems defined on compact manifolds.
\end{abstract}

\section{Introduction}

Filippov systems belong to a class of dynamical systems which are very useful to model many physical systems. The understanding of their chaotic behavior is an active area of research in dynamical systems (one may see \cite{BF1,BF2,NPV}, and the references therein).  A better comprehension of  the chaoticity of a dynamical systems can be achieved through the ergodic theory point of view. Ergodic theory deals with dynamical systems admitting an invariant measure, hence on ergodic theory one may talk about statistical properties of the dynamics. In this work we try to understand the invariant measures of Filippov systems defined on a compact manifold $M$, which are locally described by
\[
Z(p)=(F^+,F^-)_{h}:=\left\{\begin{array}{l}
F^+(p),\quad\textrm{if}\quad h(p)>0,\vspace{0.1cm}\\
F^-(p),\quad\textrm{if}\quad h(p)<0,
\end{array}\right. \quad \text{for}\quad p\in D,
\]
where $D$ is an open subset of $M$, $F^{\pm}$ are smooth vector fields on $D$, and $h:D\rightarrow\R$ is a smooth function having $0$ as a regular value  (see Section \S \ref{sec:filippov.systems}). We stress that the understanding of the invariant measures of Filippov systems is the very first step if one desires to use ergodic theory to study these systems.

In general, the Filippov solution of a discontinuous differential system passing through a point is not unique. This implies that their solutions, in general, do not enjoy the flow properties. This adds an extra difficulty when studying invariant measures. 

Hamiltonian flows preserve the volume measure, that is if $Z_t$ represents the  Hamiltonian flow, then $Vol(A) = Vol(Z_t(A))$, for any Borel set $A$. In fact if $Z_t$ is any flow and $\mu$ is probability we say that $Z_t$ preserves the probability $\mu$ if $\mu(A) = \mu(Z_t(A))$. We define analogously for the case we have a Fillipov systems. Where $Z_t(A)$ has the same usual meaning, but here it consider all the possible trajectories for the Filippov system. Full details of definitions may be found on Section \S \ref{prem}.

Our main result deals with Fillipov systems which preserves a measure equivalent to the volume measure.

\begin{main}\label{Tvol}
For $\al^{\pm}>0,$ let $f:M\rightarrow (0,\infty)$ be a piecewise constant function defined as $\alpha^{\pm}$ if $h(p)\gtrless 0$. The Filippov system $Z=(F^+,F^-)_{h}$ preserves $\nu=f\cdot \lambda$ if, and only if,  the vector fields $F^{\pm}$ preserve the measures $\nu^{\pm}=\alpha^{\pm}\cdot \lambda$ on $\Sigma^{\pm}$  and $\alpha^+F^+h(p)=\alpha^-F^-h(p)$ for every $p\in\Sigma$.
\end{main}

Theorem \ref{Tvol} is proven is Section \S \ref{sec:inv.measures}. We also provide two main consequences of Theorem \ref{Tvol}. The first consequence (Corollary \ref{c1}) shows that the Filippov systems preserving Lebesgue measure are the refractive ones which preserve Lebesgue measure in the regions of continuity. The second consequence (Corollary \ref{c2}) gives a necessary condition for a tangency-tangency point of a planar Filippov system  to be a center point. 

Now, let us grasp some ideas before trying to understand general invariant measures for Filippov systems. For instance, consider the following Filippov systems defined on an open set $U\subset\R^2$:
\[
Z_1(x,y)=\left\{\begin{array}{ll}
(1,-1),& \text{if}\quad y>0,\\
(1,1),& \text{if}\quad y<0,\\
\end{array}\right.\quad
Z_2(x,y)=\left\{\begin{array}{ll}
(1,-1),& \text{if}\quad y>0,\\
(-1,1),& \text{if}\quad y<0.\\
\end{array}\right.
\]
We know that every open set $V\subset U$ flowing through the trajectories of $Z_1$ eventually collapses on $\Sigma=\{y=0\}$. This phenomenon prevents the existence of any invariant probability measure (see Figure \ref{proto1}). Indeed, assume that $Z_1$ admits an invariant measure $\mu$. We know that there exists $t_0>0$ such that $Z_{t_0}(J_1)=Z_{t_0}(J_2)=I$. Therefore, $\mu(J_1)=\mu(I)=\mu(J_2)$. However $J_1\cup J_2\subset J=Z_{-t_0}(I)$. Hence, $\mu(I)=\mu(J)$ and $\mu(J_1)+\mu(J_2)\leq \mu(J)$ which leads to a contradiction.
\begin{figure}[h]
\begin{center}
\begin{overpic}[width=8cm]{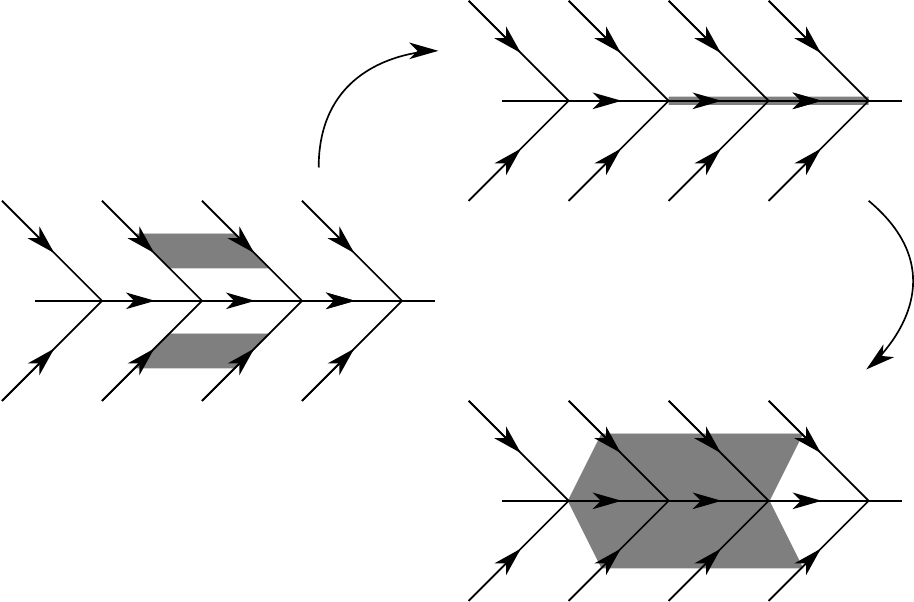}
\put(40,43){$A$}
\put(88,65){$B$}
\put(88,23){$C$}
\put(28,39){$J_1$}
\put(28,25){$J_2$}
\put(86,57){$I$}
\put(70,20){$J$}
\end{overpic}
\end{center}

\bigskip

\caption{The existence of an invariant measure for the Filippov systems $Z_1$ leads to a contradiction.}\label{proto1}
\end{figure}

This illustrates that the break of uniqueness for the solutions of Filippov systems seems to be a  barrier for the existence of invariant measures.  Hence, we conjecture the following:

\vspace{0.3cm}

\noindent\textbf{Conjecture} If the Filippov systems admits an invariant measure $\mu$, then the set formed by all possible orbits which contain points of nonuniqueness of solution has zero measure.

\vspace{0.3cm}

In what follows we shall denote by $\CN_{Z}$ the set of points in $M$ for which there is a break of uniqueness for the solutions of $Z$ and by $\sat(\CN_Z)$ the set formed by all possible orbits containing points of $\CN_Z$ (formal definitions are given in Section \S \ref{sec:diff.inclusion}).  The above statement is a very natural conjecture and in some sense the proof really goes as show in the motivational example above. Nevertheless, Filippov systems have some peculiarities and the proof has a little technical problem to be overcomed.
 In short, we need a kind of continuous selection property for its solutions. This property is held, in particular, when one assume some Lipschitz condition (see Section \S \ref{sec:diff.inclusion} and Appendix for precise definitions) which is usually not satisfied for Filippov systems. However, in order to support our conjecture, we prove a very similar result as stated in the conjecture for differential inclusions (see Appendix for the proof). We hope to motivate the reader to overcome our proof and obtain a positive answer to our conjecture.

\begin{remark}
 We are able to establish the conjecture for planar Fillippov systems and also for Filippov systems without tangencies. In these cases, the argument is basically the one given in the motivational example above. The technical difficulty mentioned above is to control the trajectories that goes through tangencies in higher dimensional Filippov systems.
\end{remark}
 
Finally, in Section \S \ref{sec:on.torus} we study several examples of Filippov systems and their invariant measures. The first example (see Section \S \ref{subsec1}) deals with constant piecewise vector fields defined on the torus $\mathbb{T}^2$ and on the Klein bottle $\mathbb{K}^2$. We provide conditions (Proposition \ref{torusklein})  in order for these systems to preserve some absolutely continuous measures. The second example (see Section \S \ref{subsec2}) is a Filippov system defined on $\mathbb{T}^2$ with no invariant probability measures and such that $\sat(\CN_Z)=\mathbb{T}^2$. The third example (see Section \S \ref{subsec3}) is a Filippov system defined on $\mathbb{T}^2$ preserving an absolutely continuous probability measure, for which $\sat(\CN_Z)$ is an open set strictly contained in $\mathbb{T}^2$. The fourth and last example (see Section \S \ref{subsec4}) is a Filippov system defined on a compact non-orientable manifold $M$ with no invariant probability measures, for which $\sat(\CN_Z)$ is a closed set strictly contained in $M$.

\section{Preliminary concepts}\label{prem}

In this section, we first introduce the concept of differential inclusions. Then, we use this concept to state the Filippov conventions for solutions of piecewise smooth systems. Finally, we define the meaning of a Filippov system to preserve a measure.

\subsection{Differential Inclusion}\label{sec:diff.inclusion}
In what follows we briefly introduce the concept of differential inclusions. For more details on this subject we recommend the books \cite{AC,S}. Let $U$ be an open subset set of $\R^n$ and $\CF:U\rightarrow \R^n$ be a set-valued function, that is, for each $x\in U$, $\CF(x)\subset \R^n$. Given an interval $I\subset\R,$ a function $\phi:I\rightarrow U$ is said to be a solution of the differential inclusion
\begin{equation}\label{di}
\dot x\in\CF(x)
\end{equation}
if $\phi$ is an absolutely continuous function satisfying \eqref{di} almost everywhere, that is $\dot\phi(t)\in\CF(\phi(t))$ for almost every $t\in I$.  Usually, given $x\in U$, $\CS_{\CF}(x)$ denotes the set of all maximal solutions $\phi(t),$ $t\in I_{x,\phi},$ of $\eqref{di}$ satisfying $\phi(0)=x$. 
Accordingly, for $x\in U$ and $V\subset U$ we denote their saturation respectively by
\[
\sat(x)=\bigcup_{\phi\in \CS_{\CF}(x)}\{\phi(t):\,t\in I_{x,\phi}\} \quad\text{and}\quad \sat(V)=\bigcup_{x\in V}\sat(x).
\]

Finally, we denote by $\CN_{\CF}$ the set of points of $x\in U$ such that $\# \CS_{\CF}(x)>1$, that is there exist at least two solutions $\phi_1, \phi_2\in \CS_{\CF}(x)$ such that $\phi_1\neq\phi_2$. In this case, since the solutions are maximal, we are able to find $t_0\in I_{x,\phi_1}\cap I_{x,\phi_2}$ for which $\phi_1(t_0)\neq\phi_2(t_0)$. In other words $\CN_{\CF}$ constitutes the set of points in $U$ for which the uniqueness of solution is lost.

\subsection{Filippov systems}\label{sec:filippov.systems}
Let $M$ be a compact Riemannian manifold. Throughout this paper we fix the metric of $M$, that is for each $p \in M$ we associate an inner product $\langle\cdot\,,\cdot\rangle_p$ on $T_pM$.  When the context is clear we shall denote $\langle\cdot\,,\cdot\rangle_p=\langle\cdot\,,\cdot\rangle.$

Let $N\subset M$ be a codimension 1 compact submanifold of $M$. Denote by $C_i$, $i=1,2,\ldots,k$, the connected components of $M\setminus N$. Let $X_i: M\rightarrow TM$, for $i=1,2,\ldots,k$, be vector fields on $M$, i.e. $X_i(p)\in T_pM$. Consider a piecewise smooth vector field on $M$ given by
\begin{equation}\label{dds}
Z(p)=X_i(p)\,\,\textrm{if}\,\, p\in C_i,\,\,\text{for}\,\, i=1,2,\ldots,k.
\end{equation}

Since $N$ is a codimension 1 compact submanifold of $M$, we can find, for each $p\in N,$ a neighborhood $D\subset M$ of $p$ and a function $h:D\rightarrow\R$, having $0$ as a regular value, such that $\Sigma=N\cap D=h^{-1}(0)$. Moreover, the neighborhood $D$ can be taken sufficiently small in order that $D\setminus \Sigma$ is composed by two disjoint region $\Sigma^+$ and $\Sigma^-$ such that $F^+=Z|_{\Sigma^+}$ and $F^-=Z|_{\Sigma^-}$ are smooth vector fields. Accordingly, the piecewise smooth vector field \eqref{dds} may be locally described as follows:
\begin{equation}\label{locdds}
Z(p)=(F^+,F^-)_{h}=\left\{\begin{array}{l}
F^+(p),\quad\textrm{if}\quad h(p)>0,\vspace{0.1cm}\\
F^-(p),\quad\textrm{if}\quad h(p)<0,
\end{array}\right. \quad \text{for}\quad p\in D.
\end{equation}

In \cite{F}, Filippov stated that the local trajectories of system \eqref{locdds} is a solution of a differential inclusion $\dot p\in\CF_Z(p)$, where $\CF_Z$ is the following set-valued function:
\begin{equation}\label{FZ}
\CF_Z(p)=\dfrac{F^+(p)+F^-(p)}{2}+\sgn(h(p))\dfrac{F^+(p)-F^-(p)}{2},
\end{equation}
and
\[
\sgn(u)=\left\{
\begin{array}{ll}
-1&\text{if}\,\, u<0,\\

[-1,1]&\text{if}\,\, u=0,\\
1&\text{if}\,\, u>0.
\end{array}\right.
\]
This approach is called Filippov's convention. The piecewise smooth vector field \eqref{dds} is called Filippov system when it is ruled by the Filippov's convention.  

We stress that for the case of Filippov systems \eqref{locdds}, the solutions of the differential inclusion $\dot p\in\CF_Z(p)$  have an easy geometrical interpretation. We shall briefly discuss it at the beginning of Section \S \ref{sec:inv.measures}. For now, we define the following open regions on $\Sigma$:
\begin{equation}\label{filippov}
\begin{array}{l}
\Sigma^c=\{p\in \Sigma:\, F^+h(p)\cdot F^-h(p) > 0\},\\

\Sigma^s=\{p\in\Sigma:\, F^+h(p)<0,\,F^-h(p) > 0\},\\

\Sigma^e=\{p\in\Sigma:\, F^+h(p)>0,\,F^-h(p) < 0\},\\
\end{array} 
\end{equation}
where $F^{\pm}h(p)=\langle\nabla h(p),F^{\pm}(p)\rangle$. Usually they are called {\it crossing}, {\it sliding}, and {\it escaping} region, respectively.

For sake of simplicity we denote by $\CS_Z(p)$ and $\CN_Z$ the sets $\CS_{\CF_Z}(p)$ and $\CN_{\CF_{Z}}$, respectively. Notice that $\Sigma^s\cup\Sigma^e\in\CN_Z$.

\subsection{Measure preserving}\label{sec:measure.preserving}

A Riemannian manifold $M$ can be seen  as a measurable space, where the sigma algebra $\mathcal B$ is the Borel sigma algebra that is, the one generated by the open sets of $M$. Throughout this paper we shall only work with Borel measures, that is the ones which $\sigma$-algebra associated is the Borel $\sigma$-algebra.  Recall that a probability measure on $M$ is a map $\mu:\mathcal B \rightarrow [0,1]$ such that $\mu(M)=1$, $\mu(\cup_{i \in \mathbb N}A_i)=\sum_i\mu(A_i)$ if the sets $A_i$ are disjoint, and $\mu(A)\leq \mu(B)$ if $A \subset B$.

Let $X_t$ denote the flow of a smooth vector field $X:M\rightarrow TM$ and $\mu$ a measure on $M$. We say that a flow $X_t$ preserves a measure $\mu$ if: for any subset Borel set $A \subset M$, $\mu(X_t(A))=\mu(A)$, $\forall t \in \mathbb R$. Nevertheless, when one consider differential inclusions and, in particular, Filippov systems, we have seen that for a given initial condition $p_0\in M$ it may exist several solutions starting at $p_0$. Consequently, the previous definition of flow and measure preserving fails. In order to overcome this difficulty, considering the analogous definition of measure preserving for flow, we say that the differential inclusions \eqref{di} preserves a measure $\mu$ if 
\begin{equation}\label{pres:di}
\mu\big(\CS_{\CF}(A)(t)\big)=\mu(A),
\end{equation}
for any Borel subset $A\subset M$, where
\[
\CS_{\CF}(A)(t)=\bigcup_{x\in A}\CS_{\CF}(x)(t).
\]
For Filippov systems we denote
\[
Z_t(p)= \CS_{Z}(p)(t)=\{\phi(t):\, \phi\in \CS_{Z}(p)\}.
\]
Hence, from \eqref{pres:di}, we say that the Filippov system \eqref{locdds} preserves a measure $\mu$ if $\mu(Z_t(A))=\mu(A),$ for any Borel subset $A\subset M$.  

Due to the nonuniqueness of solutions this concept may be a little bit restrictive for differential inclusion in general, indeed one may find different approaches to work with a measure preserving differential inclusions (e.g. \cite{artstein} and the references therein).

\section{Invariant measures for Filippov systems}\label{sec:inv.measures}

Consider the Filippov system \eqref{dds} defined on the compact Riemannian manifold $M$. The solutions of the associated differential inclusion \eqref{FZ} (and the sets $Z_t(p_0)$) are well described and fairly known in the literature (see \cite{F}). In order to state these conventions the regions on $\Sigma$ given by \eqref{filippov} must be distinguished. Notice that the points on $\Sigma$ where both vectors fields $F^+$ and $F^-$ simultaneously point outward or inward from $\Sigma$ constitute, respectively, the {\it escaping} $\Sigma^e$ and {\it sliding} $\Sigma^s$ regions, and the complement of its closure in $\Sigma$ constitutes the {\it crossing region} $\Sigma^c$. The complement of the union of those regions $\Sigma^t$ constitute the {\it tangency} points between $F^+$ or $F^-$ with $\Sigma$. For $p\in\Sigma^c$ the solutions either side of the discontinuity $\Sigma$, reaching $p$, can be joined continuously, forming a solution that crosses $\Sigma^c\subset N$. Alternatively, for $p\in \Sigma^{s,e}=\Sigma^s\cup \Sigma^e\subset N$ the solutions either side of the discontinuity $\Sigma$, reaching $p,$ can be joined continuously to solutions that slide on $\Sigma^{s,e}$ following the sliding vector field:
\begin{equation}\label{slid}
Z^s(p)= \dfrac{F^- h(p) F^+(p)- F^+ h(p) F^-(p)}{F^- h(p) - F^+ h(p) },\,\, \text{for} \,\, p\in \Sigma^{s,e}.
\end{equation}

\begin{proof}[Proof of Theorem \ref{Tvol}]
First of all, a necessary condition for $Z$ to preserve $\nu$ is that $\Sigma^s\cup\Sigma^e=\emptyset$. Indeed,  if $\Sigma^s\neq\emptyset$ (resp. $\Sigma^e\neq\emptyset$) we may find sets $A\subset M$, with positive measure, such that the forward flow (resp. backward flow) of $Z$ collapses $A$ into a set $\widetilde A\subset\Sigma^s$ (resp. $\widetilde A\subset\Sigma^e$), but since $\Sigma^s$ (resp. $\Sigma^e$) is a codimension one manifold it has zero volume measure, hence $\widetilde A$ has zero volume measure. Another important point is that the finite saturation of $\Sigma^t$ through the orbits of $Z$ has zero volume measure. So, we are not worried with this set.

Now, to say that the vector fields $F^{\pm}(x)$ and $Z(x)$ preserve, respectively, the measures $\nu^{\pm}$ and $\nu$ is equivalent to say that the vector fields $G^{\pm}(x)=\alpha^{\pm}F^{\pm}(x)$ and $G(x)=f(x)Z(x)$ preserve the Lebesgue measure $\lambda$.

Since $\Sigma=h^{-1}(0)$, with $0$ a regular value of $h$, the following map
\[
\eta:x \in \mathcal U \mapsto \dfrac{\nabla h(x)}{||\nabla h(x)||}\in T_xM
\] 
is well defined on some neighborhood $\mathcal U$ of  $\Sigma$. Notice that $\eta$ is a unit vector field on $\mathcal U$ which is normal to the codimension one manifold $\Sigma$. 

Let $\sigma$ be a small disk inside $\Sigma$. The flux $V^{\pm}(\sigma)$ of the vector fields $G^{\pm}$ through $\sigma$, that is the total amount of flow of $G^{\pm}$ passing through $\sigma$, is measured by the surface integral
\[
V^{\pm}(\sigma)=\int_\sigma\left\langle G^{\pm},\eta\right\rangle \,d\Sigma=\int_\sigma \dfrac{\left\langle G^{\pm},\nabla h\right\rangle}{||\nabla h||} \,d\Sigma=\int_\sigma \dfrac{G^{\pm}h}{||\nabla h||} \,d\Sigma,
\]
where $d\Sigma$ denotes the volume form of $\Sigma$. Since the vector fields $G^{\pm}$ preserve volume measure, the vector field $G$ will preserve volume measure if, and only if, $V^{+}(\sigma)=V^{-}(\sigma)$ for every small $\sigma\subset\Sigma$. Hence, we conclude that $Z$ preserves $\nu$ if, and only if, $\alpha^+F^+h(p)=\alpha^-F^-h(p)$ for every $p\in\Sigma$.
\end{proof}

In the remainder of this section we present two main consequences of Theorem \ref{Tvol}.

Firstly, piecewise continuous systems of kind \eqref{dds} satisfying $F^+h(p)=F^-h(p)$ constitute a well known class of Filippov systems called {\it refractive systems} (see \cite{BMT}). The next result is obtained immediately from Theorem \ref{Tvol} by taking $\alpha^{\pm}=1$. 

\begin{mcor}\label{c1}
The Filippov system $Z=(F^+,F^-)_{h}$ preserves volume measure if, and only if, $F^{\pm}$ preserve volume measure in $\Sigma^{\pm}$ and $Z$ is a refractive system. 
\end{mcor}

Now, a point $p\in\Sigma$ is called a {\it contact of multiplicity $k$} (or {\it order} $k-1$) between a vector field $F$ and $\Sigma$, if
\[
F h(p) = F^2h(p) = \ldots = F^{k-1}h(p) =0,\text{ and } F^{k} h(p)\neq 0,
\]
where the higher Lie derivative $F^n h(p)$ is recursively defined as $F^{n}h(p) =  F (F^ {n-1}h)(p),$ for $n>1.$ In addition, for the vector fields $F^+$ and $F^-$ defining the piecewise vector field \eqref{locdds}, if $(F^{\pm})^{k-1}h(p)\lessgtr0$ then it is called invisible, otherwise it is called visible. It is fairly known that if $F^{\pm}$ are planar vector fields and $p$ is an invisible  contact of even multiplicity between $F^{\pm}$ and  $\Sigma$ satisfying $\langle F^+(p),F^-(p)\rangle<0$  (see Figure \ref{tang}),  then a first return map is well defined on a small neighborhood of $p$ in $\Sigma.$ In this case, as an application of Corollary \ref{c1}, the next result provides sufficiently conditions in order to assure that $p$ is a center point, that is there exists a small neighborhood $U$ of $p$ in $M$ such that all the orbits contained in $U\setminus\{0\}$ are closed.

\begin{figure}[h]
\begin{center}
\begin{overpic}[width=4cm]{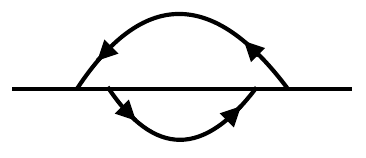}
\put(47.5,16.05){$\circ$}
\end{overpic}
\end{center}

\bigskip

\caption{Invisible contact of even multiplicity between $F^{\pm}$ and  $\Sigma$ satisfying  $\langle F^+(p),F^-(p)\rangle<0$.}\label{tang}
\end{figure}

\begin{mcor}\label{c2}
Consider the Filippov vector field $Z=(F^+,F^-)_{h}$ and let $p\in\Sigma$ be an invisible contact of even multiplicity between $F^{\pm}$ and  $\Sigma$. If $\text{tr}(d F^{\pm}(p))=0$ and $Z$ is refractive, then $p$ is a center point.
\end{mcor}

\begin{proof}
From Corollary \ref{c1}, $Z$ is a volume preserving Filippov system. The refractive condition  implies that $\langle F^+(p),F^-(p)\rangle<0$. If the first return map is not the identity then $p$ would be attractive or repulsive, which is an absurd. Therefore, the first return map is the identity which implies that $p$ is a center point.
\end{proof}

The Conjecture stated in the introduction is based on the above Theorem \ref{Tvol}. With the definitions and notations from Section \S\ref{prem},   \underline{the Conjecture can be} \underline{restated as follows}: 

\begin{quote}
{\it Suppose that a Filippov system $Z$ admits an invariant probability measure $\mu$, then $\mu(\sat(\CN_{Z}))=0$.}
\end{quote}
In the Appendix we study the invariant measures for Lipschitizian differential inclusions. In this context, Theorem \ref{Tprob.di} (from Appendix) supports the conjecture.

\section{Examples of Filippov systems defined on compact manifolds}\label{sec:on.torus}

This section is devoted to provide examples of Filippov systems defined on compact manifolds for which the main results of the previous section may be applied. In Section \S 4.1, we deal with piecewise constant vector fields on $\torus^2$ and on $\mathbb{K}^2$. As a consequence of Theorem \ref{Tvol}, it is established conditions for these systems to admit an invariant volume measure. In Section \S 4.2, it is provided an example of a Filippov system defined on $\torus^2$ such that $\CN_Z=\torus^2$ and therefore, as a consequence of our Conjecture (which is true in this case), does not admit nontrivial invariant probability measures. In Section \S 4.3 we provide an example of a Filippov system defined on $\torus^2$ such that $\torus^2\setminus \CN_Z$ is a nonempty closed set and, therefore, admits invariant probability measures. In Section \S 4.4 we provide an example of a Filippov system defined on $\torus^2$ for which $\torus^2\setminus\CN_Z$ is a nonempty open set. We show that this system admits invariant probability measures. Furthermore, we also show how to perturb this system and make a surgery in $\mathbb{T}^2$ in order to obtain another Filippov system defined on a compact manifold $M$ for which $\sat(\CN_Z)$ is still a closed set strictly contained in $M$ but now with no invariant probability measure.

First of all consider the following piecewise smooth vector field defined on the square $S=[\alpha,\alpha+p]\times[\beta,\beta+q]\subset\R^2$:
\begin{equation}\label{torus}
Z(x,y)=\left\{\begin{array}{ll}
X_i(x,y)&\text{if}\,\,\, x\in[h_i,h_{i+1}],\,\,\, \text{for}\,\,\, i=1,2,\ldots,n-1,\\
X_n(x,y) &\text{if}\,\,\, x\in[\beta,h_1],
\end{array}\right.
\end{equation}
where each $X_i(x,y)$,  $i=1,2,\ldots,n$, is a smooth vector field defined on $S$. Denote the sets of discontinuity by $\Sigma_i=[\alpha,\alpha+p]\times\{h_i\}$, for $i=1,2,\ldots,n$, with $\beta<h_1<h_2<\cdots<h_n=\beta+q$. Without loss of generality, we assume that $n$ is even (if it is not, one can virtually add another line and take the vector fields on both side as being the same).

We denote by $\mathbb{T}^2$ the Torus given by the quotient $\mathbb{T}^2=S/\sim$, where
\[
(x,y)\sim(z,w)\Leftrightarrow x-z\in p\mathbb{Z},\, y-w\in q\mathbb{Z},
\]
which identifies $[\alpha,\alpha+p]\times\{b\}$ with $[\alpha,\alpha+p]\times\{b+q\}$ and $\{a\}\times[\beta,\beta+q]$ with $\{a+p\}\times[\beta,\beta+q]$, preserving the orientation. Accordingly, the vector field \eqref{torus} can be seen as defined on $\mathbb{T}^2$. In this case the set of discontinuity $\Sigma$ is given by the union of $\Sigma_i$, for $i=1,2,\ldots,n$.  Clearly $\Sigma_n=[\alpha,\alpha+p]\times\{\beta\}=[\alpha,\alpha+p]\times\{\beta+q\}$. 

Analogously we denote by $\mathbb{K}^2$ the Klein bottle given by the quotient $\mathbb{K}^2=S/\sim$, where now
\[
(x,y)\sim(z,w)\Leftrightarrow x-z\in p\mathbb{Z},\, y+w\in q\mathbb{Z},
\]
which identifies $[\alpha,\alpha+p]\times\{\beta\}$ with $[\alpha,\alpha+p]\times\{\beta+q\}$ and $\{a\}\times[\beta,\beta+q]$ with $\{a+p\}\times[\beta,\beta+q]$, reversing the orientation in the last identification.  The piecewise vector field \eqref{torus} can be seen as defined on $\mathbb{K}^2$, but in this case an additional discontinuity is added, namely $\Sigma_0=\{a\}\times[\beta,\beta+q]=\{a+p\}\times[\beta,\beta+q]$. Thus, in order to the lines of discontinuity agree with the identification, we assume in addition that $h_1-\beta=h_{n}-h_{n-1}$ and $h_{i}-h_{i-1}=h_{n-i+1}-h_{n-i},$ for $i\in\{1,2,\ldots,n/2-1\}$. It is worth mentioning that the Klein bottle could also be obtained by reversing the orientation of the first identification.  In this case the set of discontinuity would coincide with the torus case.

\subsection{Piecewise constant vector fields on $\mathbb{T}^2$ and on $\mathbb{K}^2$}\label{subsec1}

Let $a_i\in\R$ and $b_i>0$, for $i=1,2\ldots,n$. Consider the vector field \eqref{torus} defined on $[0,1]^2$, and assume that  $X_i(x,y)=(a_i,b_i)$, for $i=1,2,\ldots,n$. For the sake of simplicity, denote $(a_{n},b_{n})=(a_{0},b_{0})$. The next result is obtained from Theorem \ref{Tvol}:

\begin{proposition}\label{torusklein}
For $\alpha_i>0,$ $i=1,2,\ldots,n,$ let $f:[0,1]^2\rightarrow\R$ be the following constant piecewise function:
\begin{equation}\label{functorus}
f(x,y)=\left\{\begin{array}{ll}
\al_i&\text{if}\,\,\, y\in[h_i,h_{i+1}],\,\,\, \text{for}\,\,\, i=1,2,\ldots,n-1,\\
\al_n &\text{if}\,\,\, y\in[0,h_1].
\end{array}\right.
\end{equation}
\begin{itemize}
\item[$(a)$]
The vector field \eqref{torus} defined on $\mathbb{T}^2$ preserves the measure $\nu=f\cdot \lambda$ if, and only if,  for some $C>0,$ $\al_i=C/b_i$, for $i\in\{1,2,\ldots, n\}$.

\item[$(b)$]
The vector field \eqref{torus} defined on $\mathbb{K}^2$ preserves the measure $\nu=f\cdot \lambda$ if, and only if, for some $C>0,$ $\al_i=C/b_i$,  for $i\in\{1,2,\ldots, n\}$, and $a_i/b_i=a_{n-i-1}/b_{n-i-1},$ for $i\in\{0,1,\ldots, n/2-1\}$.
\end{itemize}
\end{proposition}

\begin{remark}
Note that, from statement $(a)$ of Proposition \ref{torusklein}, when the piecewise vector field \eqref{torus} is defined on $\mathbb{T}^2$, one can always find a piecewise constant function \eqref{functorus} such that \eqref{torus} preserves the absolutely continuous measure $\nu=f\cdot \lambda$. Nevertheless that is not the case when \eqref{torus} is defined on $\mathbb{K}^2$. Indeed, from statement $(b)$ of Proposition \ref{torusklein}, some conditions on the parameters of \eqref{torus} must be satisfied.
\end{remark}

\begin{proof}[Proof of Proposition \ref{torusklein}]
We know that each vector field $X_i$ preserves the measure $\al_i\cdot\lambda$, for $i=1,2,\ldots,n$. Applying Theorem \ref{Tvol} for each connected component $\Sigma_i$, $i=1,2,\ldots,n$, of the discontinuity manifold $\Sigma$ we get that $Z$ preserves the measure $\nu=f\cdot \lambda$ if, and only if,
\begin{equation}\label{linear}
\left\{\begin{array}{l}
0=b_i\al_i-b_{i+1}\al_{i+1},\,\,\,\text{for}\,\,\, i=1,2,\ldots,n-1,\\
0=b_n\al_n-b_1\al_1.
\end{array}\right.
\end{equation}
The last equality of system \eqref{linear} is due to the identification $[0,1]\times\{0\} \sim[0,1]\times\{1\}$.
Adding up the first $n-1$ equalities of \eqref{linear} we get $-b_n\al_n+b_1\al_1=0$, which is equivalent to the last equality of \eqref{linear}. Therefore, the system of linear equations \eqref{linear} admits non-trivial solutions. Solving it we conclude that $(\al_1,\al_2,\ldots,\al_n)=C\big(b_1^{-1},b_2^{-1},\ldots,b_n^{-1}\big)$, for some $C>0$. It concludes the proof of statement $(a)$.

When the vector field \eqref{torus} is defined on $\mathbb{K}^2$, due to the identification we are using, one can see that system \eqref{linear} is again a necessary condition for \eqref{torus} to preserve $\nu$, hence $(\al_1,\al_2,\ldots,\al_n)=C\big(b_1^{-1},b_2^{-1},\ldots,b_n^{-1}\big)$, for some $C>0$. Moreover, applying Theorem \ref{Tvol} regarding the set of discontinuity $\Sigma_0$, we see that $\al_i a_i=\al_{n-i-1}a_{n-i-1},$ which implies $a_i/b_i=a_{n-i-1}/b_{n-i-1}$, for $i\in\{0,1,\ldots,n/2-1\}$. We conclude the proof of statement $(b)$ by noticing that the finite saturation of an orbit passing, eventually, through the intersection between $\Sigma_0$ with $\Sigma_i,$ for some $i,$  has zero volume measure.\end{proof}

\subsection{The saturation of $\CN_Z$ is the whole $\mathbb{T}^2$} \label{subsec2}
As a trivial example of a piecewise smooth system such that $\CN_Z=\mathbb{T}^2$, we may consider the following piecewise constant vector field defined on the torus $\mathbb{T}^2=[0,1]^2/\sim$:

\begin{equation}\label{exemplo1}
Z(x,y)=\left\{
\begin{array}{cl}
\left(\!\!\begin{array}{c}
0\\
-1
\end{array}\!\!\right)& \text{if}\quad y\geq0,\vspace{0.2cm}\\
\left(\begin{array}{l}
 0\\
1
\end{array}\right)& \text{if}\quad y\geq0.
\end{array}\right.
\end{equation}
Indeed, for a given $p\subset M\setminus \Sigma$ its forward trajectory reach the sliding region, and its backward trajectory reach the escaping region. As performed in the introduction, this system does not admit invariant probability measures.

\subsection{The saturation of $\CN_Z$ is open and strictly contained in $\mathbb{T}^2$}\label{subsec3} 
When $\sat(\CN_Z)\neq \mathbb{T}^2$ is open we have that $K=M\setminus \sat(\CN_Z)$ is a compact set for which $Z_t\big|_K$ is a flow. In this case, we get the existence of an invariant probability measure, since a continuous flow on a compact set always admits an invariant measure. We extend this measure to the whole space $\mathbb{T}^2$ by giving zero measure for the set $sat(\CN_Z)$. As an example of that, we may consider the vector field defined on the torus $[0,\pi]\times[-3/2,3]/\sim$ (see Figure \ref{figexample3}):
\begin{equation}\label{exemplo2}
Z(x,y)=\left\{
\begin{array}{ll}
\left(\begin{array}{c}
 1\\
\left(y-\dfrac{5}{2}\right)\left(y-\dfrac{7}{2}\right)\left(\dfrac{3}{5}-\sin^2(x)\right)
\end{array}\right)& \text{if}\quad \dfrac{3}{2}\leq y\leq 3,\vspace{0.2cm}\\

\left(\begin{array}{c}
1\\
(y-2)(y-1)\left(-\dfrac{3}{5}+\sin^2(x)\right)
\end{array}\right)& \text{if}\quad 0<y<\dfrac{3}{2},\vspace{0.2cm}\\

\left(\begin{array}{c}
1\\
(y+2)(y+1)\left(\dfrac{3}{5}-\sin^2(x)\right)
\end{array}\right)& \text{if}\quad -\dfrac{3}{2}<y<0.\vspace{0.2cm}\\
\end{array}\right.
\end{equation}

\begin{figure}[h]
\begin{center}
\begin{overpic}[width=7.5cm]{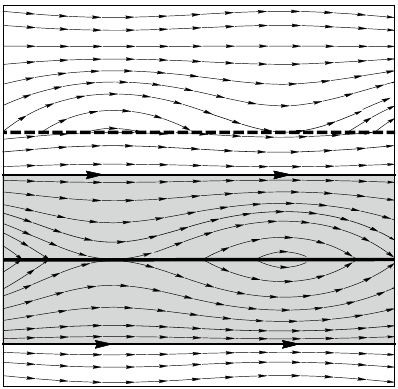}
\put(101,63){$\Sigma_2$}
\put(101,30){$\Sigma_1$}
\put(27,35){$c_1$}
\put(27,31){$\bullet$}
\put(70,35){$c_2$}
\put(70,31){$\bullet$}
\put(27,67.3){$c_3$}
\put(27,63.3){$\bullet$}
\put(70,67.3){$c_4$}
\put(70,63.3){$\bullet$}
\end{overpic}
\end{center}

\bigskip

\caption{Phase space of the piecewise smooth vector field \eqref{exemplo2} defined on the rectangle $[0,\pi]\times[-3/2,3]$. The shaded region indicates the set $\sat(\CN_Z)$.}\label{figexample3}
\end{figure}
The set of discontinuity is given by $\Sigma=\Sigma_1\cup\Sigma_2$, where $\Sigma_1=[0,\pi]\times\{0\}$ and $\Sigma_2=[0,\pi]\times\{3/2\}$. Note that the vector field is continuous on the lines $[0,\pi]\times\{-3/2\}$ and $[0,\pi]\times\{-3/2\}$. The contact between the vector field and the discontinuous manifold $\Sigma$ happens at the points 
\[
c_1=\left(0,\pi-\arcsin\left(\dfrac{3}{5}\right)\right),~ c_2=\left(0,\pi+\arcsin\left(\dfrac{3}{5}\right)\right),
\]
\[
c_3=\left(\dfrac{3}{2},\pi-\arcsin\left(\dfrac{3}{5}\right)\right),~ c_4=\left(\dfrac{3}{2},\pi+\arcsin\left(\dfrac{3}{5}\right)\right).
\]
Moreover, $\Sigma^{s,e}=\Sigma_1\setminus\{c_1,c_2\}$ and $\Sigma^c=\Sigma_2\setminus\{c_3,c_4\}$. The breaking of unicity occurs at the sliding and escaping sets and at the tangency $c_1$, so $\CN_Z=\Sigma^s\cup\Sigma^e\cup\{c_1\}$. Furthermore, it is easy to see that $\phi_1(t)=(t,1)$ and $\phi_2(t)=(t,-1)$ are limit cycles. After some simple computations we conclude that $\sat(\CN_Z)=\{(x,y)\in[0,\pi]\times[-3/2,3]:\, -1<y<1\}$, which is the open region delimited by the limit cycles $\phi_1$ and $\phi_2$.

\subsection{The saturation of $\CN_Z$ is closed and strictly contained in M}\label{subsec4} When $\sat(\CN_Z)\neq M$ is closed we may find examples for which there exist invariant probability measures as well as examples for which there are no invariant probability measures. 

Firstly we provide a piecewise smooth vector field $Z$ defined on the torus $\mathbb{T}^2=([-\pi/2,3\pi/2]\times[-3\pi/2,3\pi/2])/\sim$ (see Figure \ref{figexemple4})  for which $\sat(\CN_Z)$ is a closed set strictly contained in $\mathbb{T}^2$ and there exist invariant probability measures for $Z$. Then, we show how to perturb this system and make a surgery in $\mathbb{T}^2$ in order to obtain another system $\hat Z$ defined on a compact manifold $M$ for which $\sat(\CN_{\hat Z})$ is still a closed set strictly contained in $M$ but now $\hat Z$ does not admit invariant probability measures.

\begin{figure}[h]
\begin{center}
\begin{overpic}[width=8cm]{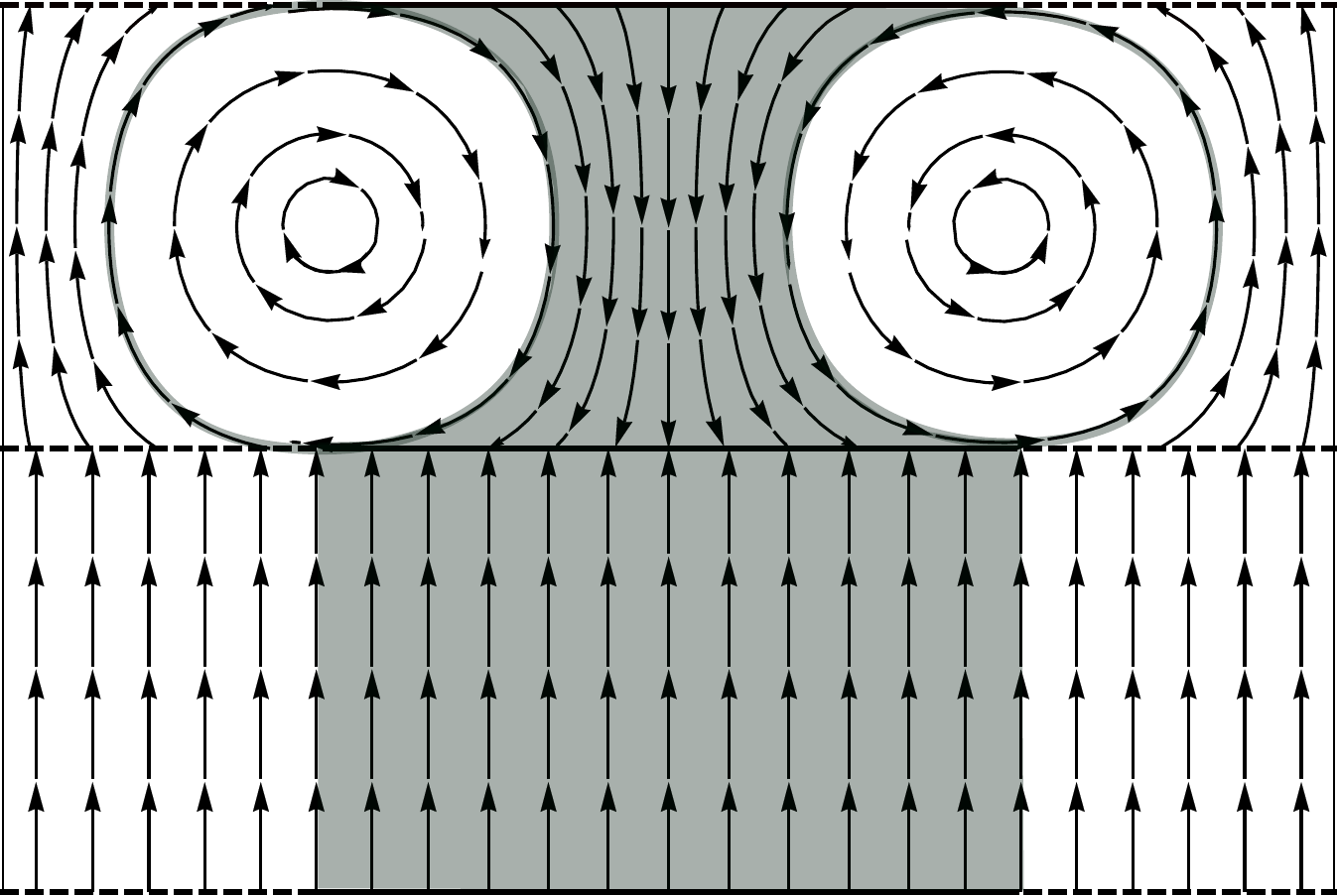}
\end{overpic}
\label{exemple4}
\end{center}

\bigskip

\caption{Phase space of the Filippov system \eqref{exemplo1} defined on the rectangle $[-\pi/2,3\pi/2]\times[-3\pi/2,3\pi/2]$. The shaded region indicates the closed set $\sat(\CN_Z)$.}\label{figexemple4}
\end{figure}

Consider the following piecewise smooth vector field defined on the torus  $([-\pi/2,3\pi/2]\times[-3\pi/2,3\pi/2])/\sim$:
\begin{equation}\label{exemplo3}
Z(x,y)=\left\{
\begin{array}{cl}
\left(\begin{array}{c}
\cos(x)(-\sqrt{3}\cos(y)+\sin(y))\\
-\sin(x)(\cos(y)+\sqrt{3}\sin(y))
\end{array}\right)& \text{if}\quad y\leq 0,\vspace{0.2cm}\\
\left(\begin{array}{l}
 0\\
1
\end{array}\right)& \text{if}\quad y\geq0.
\end{array}\right.
\end{equation}
We can easily prove that $M\setminus \sat(\CN_Z)$ is foliated by periodic orbits. On each periodic orbit in $M\setminus \sat(\CN_Z)$ we may consider the invariant equidistributed measure, hence we get the existence of invariant probability measures (once again we extended this measure to the whole space by attributing zero measure on the complement of the periodic orbit).

Now, the vector field \eqref{exemplo3} may be perturbed in order to get a system $\hat Z=(\hat F^+,\hat F^-)$ such that the following properties hold:
\begin{itemize}
\item[(i)] $\hat F^+$ has two stable limit cycles: a stable one $\gamma_1$, and an unstable one $\gamma_2$;
\item[(ii)] each limit cycle $\gamma_i$ is tangent to $\Sigma$ at two points, $c_i^1$ and $c_i^2$;
\item[(iii)] each limit cycle $\gamma_i$ encloses only one singularity $p_i$.
\item[(iv)] the $\alpha$ and $\omega$ limit sets of any point in $M\setminus \sat(\CN_Z)$ is contained in $\Gamma_2\cup\{p_1\}$ and $\Gamma_1\cup\{p_2\}$, respectively, where $\Gamma_i$ is the union of $\gamma_i$ with the arc-orbit of $\hat F^-$ connecting $c_i^1$ and $c_i^2$.
\end{itemize}

\begin{figure}[h]
\begin{center}
\begin{overpic}[width=8cm]{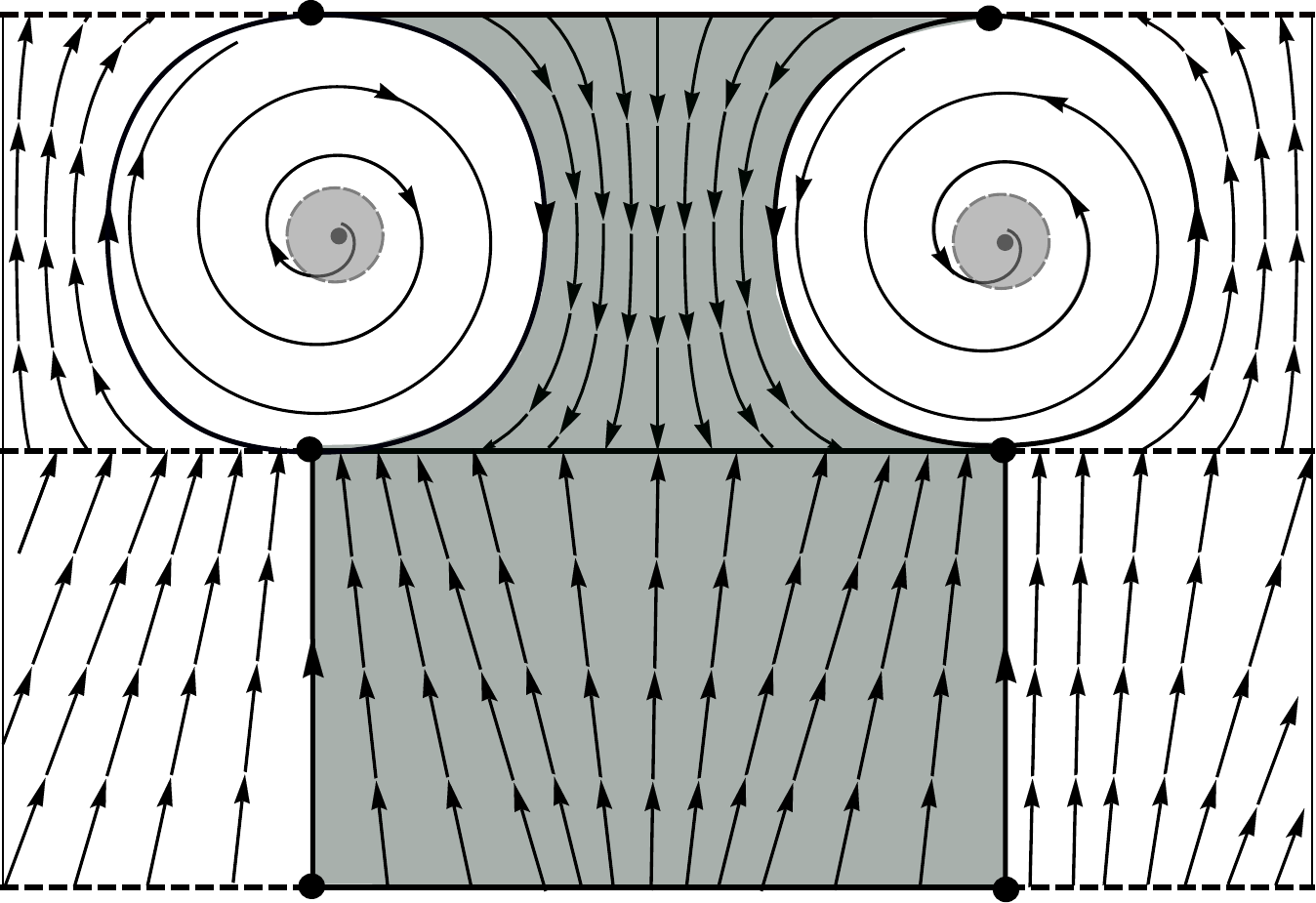}
\put(71,55){$D_2$}
\put(23,55){$D_1$}
\put(35.5,45){$\gamma_1$}
\put(60,45){$\gamma_2$}
\put(22,36.5){$c_1^1$}
\put(22,62.5){$c_1^2$}
\put(75,36.5){$c_2^1$}
\put(74,62.5){$c_2^2$}
\put(18.7,10){$\Gamma_1$}
\put(76.6,10){$\Gamma_2$}
\put(24,3){$c_1^2$}
\put(71,3){$c_2^2$}
\end{overpic}
\label{exemple5}
\end{center}

\bigskip

\caption{Phase space of a Filippov system $\hat Z$, given by a perturbation of \eqref{exemplo3}, satisfying (i)-(iv)}\label{figexemple5}
\end{figure}

Note that properties (i), (iii), and (iv) imply that $p_1$ is an unstable singularity and $p_2$ is a stable singularity. In this case, since $\gamma_i\subset\Gamma_i\subset\sat(\CN_{\hat Z}),$ the only possible invariant measures are the ones supported on the singularities $p_1$ and $p_2$. Now, we make the following surgery on $\mathbb{T}^2$: remove two small disks $D_1$ and $D_2$ centered at $p_1$ and $p_2$, respectively, and glue their boundary reversing orientation. Now, this gives a new compact manifold (non-orientable) $M$ (see Figure \ref{figexemple5}). From the stability of the singularities $p_1$ and $p_2$ we may conclude that the boundary of the disc where the surgery was performed, $\p D_1=\p D_2$, contains only crossing points. Hence, the Filippov system $\hat Z$ on this new manifold $M$ do not support any invariant measure.

\section*{Appendix: Invariant probability measures for Lipschitz differential inclusions}\label{sec:difs}

In addition to the definitions provided in Section \ref{sec:diff.inclusion},  we provide some extra definitions and a result we shall need for proving the main result of this appendix.

 Given a closed bounded interval $I\subset \R,$ with $0\in I,$ let $AC(I,\R^n)$ denote the set of all absolutely continuous function $\phi:I\rightarrow U$. For $x\in U$, denote by $\CS_{\CF}^I(x)$ the set of solutions of the differential inclusion \eqref{di} lying on $AC(I,\R^n)$ with initial condition $x.$ Notice that, $\CS_{\CF}^{I}(x)(s)=\CS_{\CF}(x)(s)$ for every $s\in I.$ Also, denote 
\[
\CS^I(\CF)=\bigcup_{x\in U}\CS^I_{\CF}(x)\subset AC(I,\R^n).
\]

In order to get some useful properties on the set-valued map $\CS^I_{\CF}:U\rightarrow AC(I,\R^n)$, some hypotheses on $\CF$ must be assumed:
\begin{itemize}
\item[(i)] $\CF(x)\subset \R^n$ is a {\it closed convex set} for every $x\in U$.

\item[(ii)] $\CF$ is {\it Lipschitizian} that is, there exists $L>0$ such that $\CF(x_1)\subset \CF(x_2)+L|x_1-x_2| B_1(0)$ for every $x_1,x_2\in U$, where $B_1(0)=\{y\in\R^n:\, |y|\leq 1\},$
\end{itemize}

Furthermore, a set-valued map $F:U\rightarrow Y$ ($Y$ topological space) is called  {\it upper semi-continuous} at $x_0\in X$ if for any open subset $W$ of $Y$ containing $F(x_0)$ there exists a neighborhood $V\subset U$ of $x_0$ such that $F(V)\subset W.$

\begin{theorem}[{\cite[Theorem 4.12]{S}}]\label{T.prel}
Assume that $\CF:U\rightarrow\R^n$ is a Lipschitizian set-valued map with closed convex values. Then, the set valued map $\CS^I_{\CF}:U\rightarrow AC(I,\R^n)$ is also Lipschitizian, in particular it is upper semi-continuous. Moreover, given $\phi_0\in\CS^I_{\CF}(x_0)$ there exists a continuous function $\Phi:U\rightarrow AC(I,\R^n)$ satisfying $\Phi(x)\in\CS^I_{\CF}(x)$ and $\Phi(x_0)=\phi_0$.
\end{theorem}

The next result implies that the Conjecture stated in the introduction holds for Lipschitizian differential inclusion.

\begin{main}\label{Tprob.di}
Suppose that $\CF:U\rightarrow\R^n$ is a Lipschitizian set-valued map with closed convex values. If the differential inclusion \eqref{di} admits an invariant probability measure $\mu$, then there exists an open set $A\subset U$ such that $\sat(\CN_{\CF})\subset A$ and $\mu(A) = 0$.
\end{main}

\begin{proof}
Assuming the existence of an invariant probability measure $\mu$ we shall prove that, for each $x_0\in\sat(\CN_{\CF})$, there exists a small neighborhood $V_{x_0}\subset U$ such that $\mu(V_{x_0})=0$. 

First assume that $x_0\in\CN_{\CF}$. Then, there exists $\phi_1,\phi_2\in\CS_{\CF}(x_0)$ and $\ov t\neq0$ such that $y_1=\phi_1(\ov t)\neq\phi_2(\ov t)=y_2$. Notice that, for $I=[0,\ov t],$ the restrictions $\phi_1\big|_I,\phi_2\big|_I\in$ belongs to $\CS^I_{\CF}(x_0).$ Applying Theorem \ref{T.prel} we get the existence of continuous functions $\Phi_1,\Phi_2:U\rightarrow \CS^I(\CF)$ such that $\Phi_i(x)\in\CS^I_{\CF}(x)$ and $\Phi_i(x_0)=\phi_i$ for $i\in\{1,2\}$. Therefore, we can find a small neighborhood $V_{x_0}\subset U$ of $x_0$ such that $\Phi_1(V_{x_0})(\ov t)\cap \Phi_2(V_{x_0})(\ov t)=\emptyset$. Denote $V_i=\Phi_i(V_{x_0})(\ov t)$. Since $V_i\subset \CS_{\CF}(V_{x_0})(\ov t)$ we have that 
\begin{equation}\label{ine1}
\mu(V_1)+\mu(V_2)\leq\mu\left(\CS_{\CF}(V_{x_0})(\ov t)\right)=\mu(V_{x_0}).
\end{equation}
Nevertheless, $V_{x_0}\subset \CS_{\CF}(V_i)(-\ov t)$ for $i=1,2$. Indeed, let $x\in V_{x_0}$, so $v_i=\Phi_i(x)(\ov t)\in V_i$ and $\psi_i(t)=\Phi_i(x)(t+\ov t)\in\CS_{\CF}(v_i),$ which implies that $x=\psi_i(-\ov t)\in  \CS_{\CF}(V_i)(-\ov t)$. Hence
\begin{equation}\label{ine2}
\mu(V_{x_0})\leq  \mu\big(\CS_{\CF}(V_i)(-\ov t)\big)=\mu(V_i),\quad \text{for}\quad i\in\{1,2\}.
\end{equation}
From \eqref{ine1} and \eqref{ine2}, we conclude that $\mu(V_{x_0})=0$.

Now, assume that $x_0\in\sat(\CN_{\CF})\setminus \CN_{\CF}$. In this case, there exist $y_0\in\CN_{\CF}$, $\phi_0\in\CS_{\CF}(y_0),$ and $t_0\neq0$ such that $\phi_0(t_0)=x_0$. Since $x_0\not\in\CN_{\CF}$ we have that $\CS_{\CF}(x_0)(t)=\{\psi_0(t)=\phi_0(t+t_0)\}$. From the first part of the proof, there exists a  neighborhood $V_{y_0}\subset U$ of $y_0$ such that $\mu(V_{y_0})=0$. Now, take $I_0=[0,t_0].$ Since $\psi_0(-t_0)=y_0,$ there exists a small neighborhood $\CV\subset AC(I_0,\R^n)$ of $\CS_{\CF}(x_0)(t)=\{\psi_0(t)\}$ such that $\phi(t_0)\in V_{y_0}$ for every $\phi\in\CV$. Now, since $\CS^{I_0}_{\CF}$ is upper semi-continuous at $x_0$ there exists a neighborhood $V_{x_0}\subset U$ of $x_0$ such that $\CS^{I_0}_{\CF}(V_{x_0})\subset \CV$, which implies that  $\CS_{\CF}(V_{x_0})(-t_0)\subset V_{y_0}$. Hence, 
\[
0\leq\mu(V_{x_0})=\mu\big(\CS_{\CF}(-t_0)\big)\leq\mu(V_{y_0})=0.
\]

Finally, take A as the following open subset of $U$:
\[
A=\bigcup \{V_x:\, x\in\sat(\CN_{\CF})\}.
\]
We conclude the proof of Theorem \ref{Tprob.di} by noticing that $\mu(A)=0$. Indeed, otherwise we would be able to find a point $\ov x\in A\cap\supp(\mu)$, which is an absurd because $\ov x\in V_{x_0}$ for some $x_0\in\sat(\CN_{\CF})$ and $\mu(V_{x_0})=0$.
\end{proof}

\section*{Acknowledgements}

We thank Marco Antonio Teixeira for his comments on a preliminary version of this article. DDN was partially supported by São Paulo Research Foundation (FAPESP) grants 2018/16430-8, 2018/13481-0, and 2019/10269-3, and by National Council for Scientific and Technological Development (CNPq) grants 401109/2016-0, 306649/2018-7, and 438975/2018-9. RV was partially supported by São Paulo Research Foundation (FAPESP) grants 2016/22475-9, 2017/06463-3, and 2018/13481-0 and by National Council for Scientific and Technological Development (CNPq) grant 401109/2016-0.

\bibliographystyle{abbrv}
\bibliography{bibliografiaNovVar2018.bib}

\end{document}